\documentclass[12pt]{article}
\usepackage{amssymb}
\usepackage{amsthm}
\usepackage{amsmath}
\usepackage{graphicx}
\usepackage{enumerate}
\usepackage{pict2e}
\usepackage{framed}

\newtheorem{theorem}{Theorem}
\newtheorem{definition}[theorem]{Definition}
\newtheorem{lemma}[theorem]{Lemma}
\newtheorem{proposition}[theorem]{Proposition}

\newtheorem{remark}[theorem]{Remark}
\newtheorem{example}[theorem]{Example}
\newtheorem{corollary}[theorem]{Corollary}
\title{Quasimodules over bounded lattices}
\author{Ivan~Chajda and Helmut~L\"anger}
\date{}

\begin{document}
	
\footnotetext{Support of the research of the first author by the Czech Science Foundation (GA\v CR), project 24-14386L, entitled ``Representation of algebraic semantics for substructural logics'', and by IGA, project P\v rF~2024~011, is gratefully acknowledged.}

\maketitle
	
\begin{abstract}
We define a quasimodule $\mathbf Q$ over a bounded lattice $\mathbf L$ in an analogous way as a module over a semiring is defined. The essential difference is that $\mathbf L$ need not be distributive. Also for quasimodules there can be introduced the concepts of inner product, orthogonal elements, orthogonal subsets, bases and closed subquasimodules. We show that the set of all closed subquasimodules forms a complete lattice having orthogonality as an antitone involution. Using the Galois connection induced by this orthogonality, we describe important properties of closed subquasimodules. We call a subquasimodule $\mathbf P$ of a quasimodule $\mathbf Q$ splitting if the sum of $P$ and its orthogonal companion is the whole set $Q$ and the intersection of $P$ and its orthogonal companion is trivial. We show that every splitting subquasimodule is closed and that its orthogonal companion is splitting, too. Our results are illuminated by several examples.
\end{abstract}
	
	

In our previous paper \cite{CL22b} we investigated semimodules over commutative semirings. Since every bounded distributive lattice is a commutative semiring the natural question arises if this study can be extended to semimodule-like structures over lattices that need not be distributive. We call such structures ``quasimodules'' and their treaty is the aim of the present paper.

We start with the definition of our main concept.

\begin{definition}
By a {\em quasimodule over a bounded lattice} $\mathbf L=(L,\vee,\wedge,0,1)$ we mean a quadruple $\mathbf Q=(Q,+,\cdot,\vec0)$ satisfying the following conditions:
\begin{enumerate}[{\rm(i)}]
\item $(Q,+,\vec0)$ is a commutative monoid,
\item $\cdot\colon L\times Q\to Q$,
\item $a(b\vec x)\approx(a\wedge b)\vec x$,
\item $0\vec x\approx\vec 0$ and $1\vec x\approx\vec x$
\end{enumerate}
where the elements of $Q$, i.e., the vectors, are denoted by $\vec x$ whereas the elements of $L$, i.e., the scalars, are denoted by Roman letters.
\end{definition}

If $\mathbf L$ is distributive then it is a commutative semiring and semimodules over commutative semirings were treated in \cite{CL22b}. If both $\mathbf L$ and $\mathbf Q$ are distributive lattices then we obtain an algebra similar to a quantale, see \cite{KP} for details. Hence in the sequel we will not assume distributivity of $\mathbf L$.

\begin{definition}
Let $\mathbf Q=(Q,+,\cdot,\vec0)$ be a quasimodule over $(L,\vee,\wedge,0,1)$ and $P\subseteq Q$. Then $(P,+,\cdot,\vec0)$ is called a {\em subquasimodule} of $\mathbf Q$ if $\vec x+\vec y,a\vec x,\vec0\in P$ for all $\vec x,\vec y\in P$ and all $a\in L$. Let $L(\mathbf Q)$ denote the set of all subquasimodules of $\mathbf Q$ and put $\mathbf L(\mathbf Q):=\big(L(\mathbf Q),\subseteq\big)$.
\end{definition}

Recall the following concepts.

Let $\mathbf L=(L,\vee,\wedge,0)$ a lattice with $0$, $M_i$ a set for all $i\in I$ and $k\in I$.
\begin{enumerate}[{\rm(i)}]
\item $\mathbf L$ is called {\em $0$-distributive} if $x,y,z\in L$ and $x\wedge z=y\wedge z=0$ together imply $(x\vee y)\wedge z=0$.
\item $p_k$ denotes the $k$-th projection from $\prod\limits_{i\in I}M_i$ to $M_k$.
\end{enumerate}

Note that the concept of a $0$-distributive lattice was mentioned in \cite{RR}. That a $0$-distributive lattice need not be distributive is shown in the following example.

\begin{example}\label{ex2}
The non-modular lattice $\mathbf N_5=(N_5,\vee,\wedge)$ visualized in Figure~1 is $0$-distributive.

\vspace*{-3mm}
	
\begin{center}
\setlength{\unitlength}{7mm}
\begin{picture}(4,8)
\put(2,1){\circle*{.3}}
\put(1,4){\circle*{.3}}
\put(3,3){\circle*{.3}}
\put(3,5){\circle*{.3}}
\put(2,7){\circle*{.3}}
\put(2,1){\line(-1,3)1}
\put(2,1){\line(1,2)1}
\put(3,3){\line(0,1)2}
\put(2,7){\line(-1,-3)1}
\put(2,7){\line(1,-2)1}
\put(1.85,.3){$0$}
\put(3.4,2.85){$a$}
\put(.35,3.85){$b$}
\put(3.4,4.85){$c$}
\put(1.85,7.4){$1$}
\put(1.2,-.75){{\rm Fig.~1}}
\put(-3,-1.75){Non-modular $0$-distributive lattice $\mathbf N_5$}
\end{picture}
\end{center}
\end{example}

\vspace*{10mm}

One method how to construct quasimodules in a natural way is described in the following lemma whose proof is straightforward.

\begin{lemma}\label{lem2}
Let $\mathbf L=(L,\vee,\wedge,0,1)$ be a bounded lattice, $I$ a non-empty index set, for every $i\in I$ let $L_i$ be a non-empty ideal of $\mathbf L$, put $Q:=\prod\limits_{i\in I}L_i$ and $\vec0:=(0,\ldots,0)$ and define $+\colon Q\times Q\to Q$ and $\cdot\colon L\times Q\to Q$ by
\begin{align*}
(a_i;i\in I)+(b_i;i\in I) & :=(a_i\vee b_i;i\in I), \\
            c(a_i;i\in I) & :=(c\wedge a_i;i\in I)
\end{align*}
for all $(a_1;i\in I),(b_i;i\in I)\in Q$ and all $c\in L$. Then $\prod\limits_{i\in I}\mathbf L_i:=(Q,+,\cdot,\vec0)$ is a quasimodule over $\mathbf L$.
\end{lemma}

\begin{definition}
We call the quasimodule $\prod\limits_{i\in I}\mathbf L_i$ over $\mathbf L$ constructed in Lemma~\ref{lem2} the {\em canonical quasimodule} $\prod\limits_{i\in I}\mathbf L_i$ over $\mathbf L$.
\end{definition}

It is worth noticing that such a canonical quasimodule is not a quantale provided $\mathbf L$ is not distributive.

It is evident that $\mathbf L(\mathbf Q)$ forms a complete lattice since $L(\mathbf Q)$ is closed with respect to arbitrary intersections.

In the following let,  let $\bigvee\limits_{j\in J}P_j$ denote the supremum of the subquasimodules $P_j,j\in J,$ of $\mathbf Q$ in $\mathbf L(\mathbf Q)$.

\begin{definition}
Let $\mathbf Q=(Q,+,\cdot,\vec0)$ be a quasimodule and $A\subseteq Q$. Then $\langle A\rangle$ denotes the smallest subquasimodule of $\mathbf Q$ including $A$. This subquasimodule is called the {\em subquasimodule} of $\mathbf Q$ {\em generated} by $A$. The set $A$ is called a {\em generating set} of $\mathbf Q$ if $\langle A\rangle=Q$. A minimal generating set of $\mathbf Q$ is called a {\em basis} of $\mathbf Q$.
\end{definition}

It is clear that finite (sub-)quasimodules always have a basis. For certain canonical quasimodules the next theorem describes a possibly infinite subquasimodule having a basis.

\begin{theorem}\label{th1}
Let $\mathbf L=(L,\vee,\wedge,0,1)$ be a bounded lattice and $q_i\in L$ for all $i\in I$. Further, let $\mathbf Q=(Q,+,\cdot,\vec0)=\prod\limits_{i\in I}[0,q_i]$ be a canonical quasimodule over $\mathbf L$ and let $P$ denote the set of all elements of $Q$ having only finitely many non-zero components. For any $i\in I$ let $\vec b_i$ denote the element of $P$ satisfying
\[
p_j(\vec b_i)=\left\{
\begin{array}{ll}
q_i & \text{if }j=i, \\
0   & \text{otherwise}
\end{array}
\right.
\]
and put $B:=\{\vec b_i|i\in I\}$. Then $\mathbf P:=(P,+,\cdot,\vec0)$ is a subquasimodule of $\mathbf Q$ and $B$ a basis of $\mathbf P$.
\end{theorem}

\begin{proof}
That $\mathbf P$ is a subquasimodule of $\mathbf Q$ including $B$ is easy to see. Let $\vec a\in P$. Then there exists some finite subset $F$ of $I$ satisfying $p_i(\vec a)=0$ for all $i\in I\setminus F$. Now $\vec a=\sum\limits_{i\in F}p_i(\vec a)\vec b_i$ showing $B$ to be a generating set of $\mathbf P$. Similarly, one can see that for every $k\in I$ we have
\[
\langle\{\vec b_i\mid i\in I\setminus\{k\}\}\rangle=\{\vec x\in P\mid p_k(\vec x)=0\}
\]
and hence $B$ is a minimal generating set, i.e., a basis of $\mathbf P$.
\end{proof}

For canonical quasimodules we can introduce the concept of inner product and orthogonality.

\begin{definition}
Let $(Q,+,\cdot,\vec0)=\prod\limits_{i\in I}\mathbf L_i$ be a canonical quasimodule over a complete lattice and $A\subseteq Q$. Define the {\em inner product} of $(a_i;i\in I),(b_i;i\in I)\in Q$ by
\[
(a_i;i\in I)(b_i;i\in I):=\bigvee_{i\in I}(a_i\wedge b_i)
\]
and a binary relation $\perp$ on $Q$, the so-called orthogonality, by
\[
\vec x\perp\vec y\text{ if and only if }\vec x\vec y=0
\]
{\rm(}$\vec x,\vec y\in Q${\rm)}. Moreover, put
\[
A^\perp:=\{\vec x\in Q\mid\vec x\perp\vec y\text{ for all }\vec y\in A\},
\]
especially,
\[
\vec y^\perp:=\{\vec x\in Q\mid\vec x\perp\vec y\}
\]
{\rm(}$\vec y\in Q${\rm)}. The set $A$ is called {\em orthogonal} if its elements are pairwise orthogonal, and it is called an {\em orthogonal basis} if it is both orthogonal and a basis.
\end{definition}

\begin{remark}
Obviously, $(a_i;i\in I)\perp(b_i;i\in I)$ if and only if $a_i\wedge b_i=0$ for all $i\in I$. Hence orthogonality of two elements of $Q$ can be defined also in the case when $\mathbf L$ is not complete.
\end{remark}

\begin{lemma}
Let $\mathbf L=(L,\vee,\wedge,0,1)$ be a bounded lattice and $q_i\in L$ for all $i\in I$. Further, let $\mathbf Q=(Q,+,\cdot,\vec0)=\prod\limits_{i\in I}[0,q_i]$ be a canonical quasimodule over $\mathbf L$, let $\vec x,\vec y\in Q$ and assume $\vec x\vec z=\vec y\vec z$ for all $\vec z\in Q$. Then $\vec x=\vec y$.
\end{lemma}

\begin{proof}
If for all $i\in I$ the element $\vec b_i$ of $Q$ is defined as in Theorem~\ref{th1} then $p_i(\vec x)=\vec x\vec b_i=\vec y\vec b_i=p_i(\vec y)$ for all $i\in I$, that is, $\vec x=\vec y$.
\end{proof}

Now we prove that for a canonical quasimodule $\mathbf Q=(Q,+,\cdot,\vec0)$ over a $0$-distributive bounded lattice and for a subset $A$ of $Q$ we have that $(A^\perp,+,\cdot,\vec0)$ is a subquasimodule of $\mathbf Q$.

\begin{proposition}\label{prop2}
Let $\mathbf Q=(Q,+,\cdot,\vec0)=\prod\limits_{i\in I}\mathbf L_i$ be a canonical quasimodule over $(L,\vee,\wedge,0,$ $1)$, assume every $\mathbf L_i$ to be $0$-distributive and let $A\subseteq Q$. Then $(A^\perp,+,\cdot,\vec0)$ is a subquasimodule of $\mathbf Q$.
\end{proposition}

\begin{proof}
Let $\vec x,\vec y\in A^\perp$, say $\vec x=(a_i;i\in I)$ and $\vec y=(b_i;i\in I)$, and $\vec z\in A$, say $\vec z=(c_i;i\in I)$. Then $a_i\wedge c_i=b_i\wedge c_i=0$ for all $i\in I$. Since every $\mathbf L_i$ is $0$-distributive we conclude $(a_i\vee b_i)\wedge c_i=0$ for all $i\in I$ and hence $\vec x+\vec y=(a_i\vee b_i;i\in I)\perp\vec z$. Since $\vec z$ was an arbitrary element of $A$ we obtain $\vec x+\vec y\in A^\perp$. Now let $a\in L$. Since $a_i\wedge c_i=0$ for all $i\in I$ we have $(a\wedge a_i)\wedge c_i=0$ for all $i\in I$ showing $a\vec x=(a\wedge a_i;i\in I)\perp\vec z$. Because $\vec z$ was an arbitrary element of $A$ we conclude $a\vec x\in A^\perp$. Finally, $0\wedge c_i=0$ for all $i\in I$ and hence $\vec0\in A^\perp$.
\end{proof}

That the condition that all $\mathbf L_i$ should be $0$-distributive cannot be omitted from Proposition~\ref{prop2} is demonstrated by the following example.

\begin{example}
The modular non-distributive lattice $\mathbf M_3=(M_3,\vee,\wedge)$ depicted in Figure~2 is not $0$-distributive.

\vspace*{-3mm}

\begin{center}
\setlength{\unitlength}{7mm}
\begin{picture}(6,6)
\put(3,1){\circle*{.3}}
\put(1,3){\circle*{.3}}
\put(3,3){\circle*{.3}}
\put(5,3){\circle*{.3}}
\put(3,5){\circle*{.3}}
\put(3,1){\line(-1,1)2}
\put(3,1){\line(0,1)4}
\put(3,1){\line(1,1)2}
\put(3,5){\line(-1,-1)2}
\put(3,5){\line(1,-1)2}
\put(2.85,.3){$0$}
\put(.35,2.85){$a$}
\put(3.4,2.85){$b$}
\put(5.4,2.85){$c$}
\put(2.85,5.4){$1$}
\put(2.2,-.75){{\rm Fig.~2}}
\put(-1.7,-1.75){Modular non-distributive lattice $\mathbf M_3$}
\end{picture}
\end{center}

\vspace*{10mm}

since $a\wedge c=b\wedge c=0$, but $(a\vee b)\wedge c=1\wedge c=c\ne0$. Consider the canonical quasimodule $\mathbf Q:=\mathbf M_3\times[0,a]$ over $\mathbf M_3$ and put $P:=[0,a]\times[0,a]$. Then $P\in L(\mathbf Q)$. Of course, $[0,a]$ is $0$-distributive. Now $P^\perp\notin L(\mathbf Q)$ since $(b,0),(c,0)\in P^\perp$, but $(b,0)+(c,0)=(b\vee c,0)=(1,0)\notin P^\perp$ because $(1,0)\not\perp(a,0)\in P$. Observe that $B:=\{(0,a),(a,0),(b,0)\}$ is an orthogonal basis of $\mathbf Q$. First of all $B$ is a generating set since $(1,0)=(a,0)+(b,0)$, $(x,0)=x(1,0)$ for all $x\in M_3$ and $(x,a)=(x,0)+(0,a)$ for all $x\in M_3$. Moreover, it is easy to see that
\begin{align*}
\langle\{(0,a),(a,0)\}\rangle & =[0,a]\times[0,a]\subsetneqq M_3\times[0,a], \\	
\langle\{(0,a),(b,0)\}\rangle & =[0,b]\times[0,a]\subsetneqq M_3\times[0,a], \\	
\langle\{(a,0),(b,0)\}\rangle & =M_3\times\{0\}\subsetneqq M_3\times[0,a].
\end{align*}
\end{example}

In the following we are interested in subquasimodules closed with respect to the operator $^{\perp\perp}$. For this purpose we define

\begin{definition}
Let $\mathbf Q=(Q,+,\cdot,\vec0)$ be a canonical quasimodule. A subset $A$ of $Q$ is called {\em closed} if $A^{\perp\perp}=A$. Let $L_C(\mathbf Q)$ denote the set of all closed subquasimodules of $\mathbf Q$ and put $\mathbf L_C(\mathbf Q):=\big(L_C(\mathbf Q),\subseteq\big)$.
\end{definition}

It is easy to verify assertions in the following remark.

\begin{remark}\label{rem1}
Let $(Q,+,\cdot,\vec0)$ be a canonical quasimodule. Then $(\,^\perp,^\perp)$ is the Galois correspondence between $(2^Q,\subseteq)$ and $(2^Q,\subseteq)$ induced by $\perp$. From this we obtain for all $A,B\subseteq Q$
\begin{enumerate}[{\rm(i)}]
\item $A\subseteq A^{\perp\perp}$,
\item $A\subseteq B$ implies $B^\perp\subseteq A^\perp$,
\item $A^{\perp\perp\perp}=A^\perp$,
\item $A\subseteq B^\perp$ if and only if $B\subseteq A^\perp$.
\end{enumerate}
\end{remark}

In the following we list some elementary properties of orthogonality.

\begin{lemma}\label{lem4}
Let $(Q,+,\cdot,\vec0)$ be a canonical quasimodule and $A,A_j\subseteq Q$ for all $j\in J$ with $A\ne\emptyset$. Then the following hold:
\begin{enumerate}[{\rm(i)}]
\item $\bigcap\limits_{j\in J}A_j^\perp=\left(\bigcup\limits_{j\in J}A_j\right)^\perp$,
\item $\left(\bigcap\limits_{j\in J}A_j\right)^{\perp\perp}\subseteq\bigcap\limits_{j\in J}A_j^{\perp\perp}$,
\item $Q^\perp=\{\vec0\}$,
\item $A\cap A^\perp=\{\vec0\}$,
\item $\{\vec0\}^\perp=Q$.
\end{enumerate}
\end{lemma}

\begin{proof}
\
\begin{enumerate}[(i)]
\item For all $\vec a\in Q$ the following are equivalent:
\begin{align*}
\vec a & \in\bigcap\limits_{j\in J}A_j^\perp, \\
\vec a & \in A_j^\perp\text{ for all }j\in J, \\
\vec a & \perp\vec x\text{ for all }j\in J\text{ and all }\vec x\in A_j, \\
\vec a & \perp\vec x\text{ for all }\vec x\in\bigcup_{j\in J}A_j, \\
\vec a & \in\left(\bigcup\limits_{j\in J}A_j\right)^\perp.
\end{align*}
\item If $k\in J$ then we have $\bigcap\limits_{j\in J}A_j\subseteq A_k$ and hence $\left(\bigcap\limits_{j\in J}A_j\right)^{\perp\perp}\subseteq A_k^{\perp\perp}$ according to (ii) of Remark~\ref{rem1} which implies $\left(\bigcap\limits_{j\in J}A_j\right)^{\perp\perp}\subseteq\bigcap\limits_{j\in J}A_j^{\perp\perp}$.
\item follows from the fact that $\vec x\in Q$ and $\vec x\perp\vec x$ together imply $\vec x=\vec0$.
\item follows from the same fact as (iii).
\item is straightforward.
\end{enumerate}
\end{proof}

If, moreover, all sublattices $\mathbf L_i$ of the lattice $\mathbf L$ are considered to be $0$-distributive then we can prove the following.

\begin{theorem}\label{th2}
Let $\mathbf Q=(Q,+,\cdot,\vec0)=\prod\limits_{i\in I}\mathbf L_i$ be a canonical quasimodule with $0$-distributive $\mathbf L_i$, assume $A\subseteq Q$, $B\subseteq C\subseteq Q$ and $C$ to be orthogonal and let $P_j\in L_C(Q)$ for all $j\in J$. Then the following hold:
\begin{enumerate}[{\rm(i)}]
\item $L_C(\mathbf Q)=\{D^\perp\mid D\subseteq Q\}$,
\item $A^{\perp\perp}$ is the smallest closed subquasimodule of $\mathbf Q$ including $A$,
\item $\bigvee\limits_{j\in J}^cP_j=\left(\bigcup\limits_{j\in J}P_j\right)^{\perp\perp}$ where $\bigvee\limits_{j\in J}^cP_j$ denotes the supremum of the closed subquasimodules $P_j,j\in J,$ of $\mathbf Q$ in $\mathbf L_C(\mathbf Q)$,
\item $^\perp$ is an antitone involution on $\mathbf L_C(\mathbf Q)$,
\item $\mathbf L_C(\mathbf Q)$ is a complete lattice,
\item $A^\perp=\langle A\rangle^\perp$,
\item $\langle C\setminus B\rangle\subseteq\langle B\rangle^\perp$.
\end{enumerate}
\end{theorem}

\begin{proof}
\
\begin{enumerate}[(i)]
\item follows from Proposition~\ref{prop2} and from (iii) of Remark~\ref{rem1}.
\item follows from (i) and (ii) of Remark~\ref{rem1} and from (i).
\item According to (ii), $\left(\bigcup\limits_{j\in J}P_j\right)^{\perp\perp}$ is the smallest closed subquasimodule of $\mathbf Q$ including $\bigcup\limits_{j\in J}P_j$ and hence also the smallest closed subquasimodule of $\mathbf Q$ including all $P_j,j\in J$.
\item follows from (ii) and (iii) of Remark~\ref{rem1} and from (i).
\item follows from (i) of Lemma~\ref{lem4} and from (i).
\item Because of $A\subseteq\langle A\rangle$ and (ii) of Remark~\ref{rem1} we have $\langle A\rangle^\perp\subseteq A^\perp$. On the other hand, $A\subseteq A^{\perp\perp}\in L(\mathbf Q)$ according to (i) of Remark~\ref{rem1} and Proposition~\ref{prop2} and hence $\langle A\rangle\subseteq A^{\perp\perp}$ which implies $A^\perp\subseteq\langle A\rangle^\perp$ by (ii) and (iii) of Remark~\ref{rem1}.
\item According to the orthogonality of $C$, (vi) and Proposition~\ref{prop2} we have
\[
C\setminus B\subseteq B^\perp=\langle B\rangle^\perp\in L(\mathbf Q)
\]
and hence $\langle C\setminus B\rangle\subseteq\langle B\rangle^\perp$.
\end{enumerate}
\end{proof}

\begin{lemma}\label{lem6}
Let $\mathbf Q=\prod\limits_{i\in I}\mathbf L_i$ with $\mathbf L_i=(L_i,\vee,\wedge,0)$ be a canonical quasimodule over $(L,\vee,\wedge,0)$ and $M_i\subseteq L_i$ for all $i\in I$. Then the following hold:
\begin{enumerate}[{\rm(i)}]
\item If $P\in L(\mathbf Q)$ then $p_i(P)\in L(\mathbf L_i)$ for all $i\in I$,
\item $\prod\limits_{i\in I}M_i\in L(\mathbf Q)$ if and only if $M_i\in L(\mathbf L_i)$ for all $i\in I$.
\end{enumerate}
\end{lemma}

\begin{proof}
\
\begin{enumerate}[(i)]
\item Let $j\in I$, $a,b\in p_j(P)$ and $c\in L$. Then there exist $\vec x,\vec y\in P$ with $p_j(\vec x)=a$ and $p_j(\vec y)=b$. Now we have $a\vee b=p_j(\vec x+\vec y)\in p_j(P)$ and $c\wedge a=p_j(c\vec x)\in p_j(P)$. Moreover $0=p_j(\vec0)\in p_j(P)$.
\item If $\prod\limits_{i\in I}M_i\in L(\mathbf Q)$ then $M_i\in L(\mathbf L_i)$ for all $i\in I$ according to (i). Now assume $M_i\in L(\mathbf L_i)$ for all $i\in I$ and let $\vec x,\vec y\in\prod\limits_{i\in I}M_i$ and $c\in L$. Then $p_i(\vec x),p_i(\vec y)\in M_i$ for all $i\in I$. From this we obtain $p_i(\vec x)\vee p_i(\vec y),c\wedge p_i(\vec x)\in M_i$ for all $i\in I$ and hence $\vec x+\vec y=\big(p_i(\vec x)\vee p_i(\vec y);i\in I\big)\in\prod\limits_{i\in I}M_i$ and $c\vec x=\big(c\wedge p_i(\vec x);i\in I\big)\in\prod\limits_{i\in I}M_i$. Finally, $0\in M_I$ for all $i\in I$ and hence $\vec0=(0;i\in I)\in\prod\limits_{i\in I}M_i$.
\end{enumerate}
\end{proof}

We show that the set of all elements being orthogonal to a fixed subset of a canonical quasimodule is a product of its projections.

\begin{lemma}\label{lem1}
Let $\mathbf Q=(Q,+,\cdot,\vec0)=\prod\limits_{i\in I}\mathbf L_i$ be a canonical quasimodule and $P\subseteq Q$. Then $P^\perp=\prod\limits_{i\in I}\big(p_i(P)\big)^\perp$.
\end{lemma}

\begin{proof}
For any $\vec a\in Q$ the following are equivalent:
\begin{align*}
     \vec a & \in P^\perp, \\
     \vec a & \perp\vec x\text{ for all }\vec x\in P, \\
p_i(\vec a) & \perp p_i(\vec x)\text{ for all }\vec x\in P\text{ and all }i\in I, \\
p_i(\vec a) & \perp y\text{ for all }i\in I\text{ and all }y\in p_i(P), \\
p_i(\vec a) & \in\big(p_i(P)\big)^\perp\text{ for all }i\in I, \\
     \vec a & \in\prod_{i\in I}\big(p_i(P)\big)^\perp.	
\end{align*}
\end{proof}

If all sublattices $\mathbf L_i$ of the lattice $\mathbf L$ are $0$-distributive then the projections of a closed subquasimodule of a canonical quasimodule over $\mathbf L$ are again closed subquasimodules of the canonical quasimodules $\mathbf L_i$, see the following result.

\begin{theorem}\label{th3}
Let $\mathbf Q=(Q,+,\cdot,\vec0)=\prod\limits_{i\in I}\mathbf L_i$ be a canonical quasimodule with $0$-distributive $\mathbf L_i=(L_i,\vee,\wedge,0)$ and $P\subseteq Q$. Then $(P,+,\cdot,\vec0)$ is a closed subquasimodule of $\mathbf Q$ if and only if for every $i\in I$ there exists some closed subquasimodule $(P_i,\vee,\wedge,0)$ of $(L_i,\vee,\wedge,0)$ with $\prod\limits_{i\in I}P_i=P$.
\end{theorem}

\begin{proof}
First assume $(P,+,\cdot,\vec0)$ to be a closed subquasimodule of $\mathbf Q$. According to Lemma~\ref{lem1} we have
\begin{align*}
       P^\perp & =\prod_{i\in I}\big(p_i(P)\big)^\perp, \\
P^{\perp\perp} & =\prod_{i\in I}\big(p_i(P^\perp)\big)^\perp=\prod_{i\in I}\big(p_i(P)\big)^{\perp\perp}.
\end{align*}
Because of
\[
P^{\perp\perp}=P\subseteq\prod_{i\in I}p_i(P)\subseteq\prod_{i\in I}\big(p_i(P)\big)^{\perp\perp}
\]
according to (i) of Remark~\ref{rem1} we conclude
\[
P=\prod_{i\in I}p_i(P)=\prod_{i\in I}\big(p_i(P)\big)^{\perp\perp}
\]
which together with $p_i(P)\subseteq\big(p_i(P)\big)^{\perp\perp}$ for all $i\in I$ (again according to (i) of Remark~\ref{rem1}) implies $p_i(P)=\big(p_i(P)\big)^{\perp\perp}$ for all $i\in I$, i.e., $p_i(P)$ is a closed subset of $L_i$ for all $i\in I$. Because of Lemma~\ref{lem6} we conclude that $(p_i(P),\vee,\wedge,0)$ is a closed subquasimodule of $(L_i,\vee,\wedge,0)$ for all $i\in I$. If, conversely, for every $i\in I$ there exists some closed subquasimodule $(P_i,\vee,\wedge,0)$ of $(L_i,\vee,\wedge,0)$ with $\prod\limits_{i\in I}P_i=P$ then
\begin{align*}
       P^\perp & =\prod_{i\in I}\big(p_i(P)\big)^\perp=\prod_{i\in I}P_i^\perp, \\
P^{\perp\perp} & =\prod_{i\in I}\big(p_i(P^\perp)\big)^\perp=\prod_{i\in I}P_i^{\perp\perp}=\prod_{i\in I}P_i=P,
\end{align*}
i.e., $P$ is a closed subset of $Q$ and, moreover, $(P,+,\cdot,\vec0)$ is a closed subquasimodule of $\mathbf Q$ because of Lemma~\ref{lem6}.
\end{proof}

For the lattice of closed subquasimodules, the following result plays an essential role.

\begin{corollary}\label{cor1}
Let $\mathbf Q=\prod\limits_{i\in I}\mathbf L_i$ be a canonical quasimodule with $0$-distributive $\mathbf L_i$. Then $\mathbf L_C(\mathbf Q)\cong\prod\limits_{i\in I}\mathbf L_C(\mathbf L_i)$.
\end{corollary}

\begin{proof}
From Theorem~\ref{th3} we see that the mapping $(P_i;i\in I)\mapsto\prod\limits_{i\in I}P_i$ is a bijection from $\prod\limits_{i\in I}L_C(\mathbf L_i)$ to $L_C(\mathbf Q)$. Since for $(P_i;i\in I),(Q_i;i\in I)\in\prod\limits_{i\in I}L_C(\mathbf L_i)$ we have $\prod\limits_{i\in I}P_i\subseteq\prod\limits_{i\in I}Q_i$ if and only if $P_i\subseteq Q_i$ for all $i\in I$, we conclude that this bijection is an isomorphism from $\prod\limits_{i\in I}\mathbf L_C(\mathbf L_i)$ to $\mathbf L_C(\mathbf Q)$.
\end{proof}

We illustrate the concepts and results investigated above in the following example.

\begin{example}\label{ex1}
The lattice $\mathbf N_5$ from Example~\ref{ex2} can be considered as a canonical quasimodule over $\mathbf N_5$. Since $A^\perp=\bigcap\limits_{\vec x\in A}\vec x^\perp$ for all $A\subseteq N_5$ and
\begin{align*}
0^\perp & =\{0,a,b,c,1\}, \\
a^\perp & =\{0,b\}, \\
b^\perp & =\{0,a,c\}, \\
c^\perp & =\{0,b\}, \\
1^\perp & =\{0\},
\end{align*}
we have $L_C(\mathbf N_5)=\{\{0\},\{0,b\},\{0,a,c\},\{0,a,b,c,1\}\}$ and hence $\mathbf L_C(\mathbf N_5)$ is the four-element Boolean algebra. From this and Corollary~\ref{cor1} we conclude that for any positive integer $n$, $\mathbf L_C(\mathbf N_5^n)$ is the $2^{2n}$-element Boolean algebra. Further, the canonical quasimodule $\mathbf Q:=\mathbf N_5\times[0,a]$ over $\mathbf N_5$ has the following subquasimodules:
\begin{align*}
   P_1 & =\{(0,0)\}, \\
   P_2 & =\{(0,0),(0,a)\}, \\
   P_3 & =\{(0,0),(a,0)\}, \\
   P_4 & =\{(0,0),(a,a)\}, \\
   P_5 & =\{(0,0),(b,0)\}, \\
   P_6 & =\{(0,0),(0,a),(a,a)\}, \\
   P_7 & =\{(0,0),(a,0),(a,a)\}, \\
   P_8 & =\{(0,0),(a,0),(c,0)\}, \\
   P_9 & =\{(0,0),(a,a),(c,a)\}, \\
P_{10} & =\{(0,0),(0,a),(a,0),(a,a)\}, \\
P_{11} & =\{(0,0),(0,a),(a,a),(c,a)\}, \\
P_{12} & =\{(0,0),(0,a),(b,0),(b,a)\}, \\
P_{13} & =\{(0,0),(0,a),(a,0),(a,a),(c,a)\}, \\
P_{14} & =\{(0,0),(a,0),(a,a),(c,0),(c,a)\}, \\
P_{15} & =\{(0,0),(a,0),(b,0),(c,0),(1,0)\}, \\
P_{16} & =\{(0,0),(a,a),(b,0),(c,a),(1,a)\}, \\
P_{17} & =\{(0,0),(0,a),(a,0),(a,a),(c,0),(c,a)\}, \\
P_{18} & =\{(0,0),(0,a),(a,a),(b,0),(b,a),(c,a),(1,a)\}, \\
P_{19} & =\{(0,0),(a,0),(a,a),(b,0),(c,0),(c,a),(1,0),(1,a)\}, \\
P_{20} & =\{(0,0),(0,a),(a,0),(a,a),(b,0),(b,a),(c,0),(c,a),(1,0),(1,a)\}.
\end{align*}
Moreover, $\{(0,a),(1,0)\}$ and $\{(0,a),(b,0),(c,0)\}$ are orthogonal bases of $\mathbf Q$ since
\begin{align*}
\langle\{(0,a),(1,0)\}\rangle & =P_{20}, \\
      \langle\{(0,a)\}\rangle & =P_2, \\
      \langle\{(1,0)\}\rangle & =P_{15}
\end{align*}
for the first one and
\begin{align*}
\langle\{(0,a),(b,0),(c,0)\}\rangle & =P_{20}, \\
	  \langle\{(0,a),(b,0)\}\rangle & =P_{12}, \\
	  \langle\{(0,a),(c,0)\}\rangle & =P_8, \\
	  \langle\{(b,0),(c,0)\}\rangle & =P_{15}.
\end{align*}
for the second one. Hence, a quasimodule can have more than one basis and these bases may be of different cardinality contrary to the case of vector spaces.

According to the description of $L_C(\mathbf N_5)$ above, Theorem~\ref{th3} and Corollary~\ref{cor1} we have $L_C(\mathbf Q)=\{P_1,P_2,P_5,P_8,P_{12},P_{15}.P_{17},$ $P_{20}\}$ and $\mathbf L_C(\mathbf Q)$ is the eight-element Boolean algebra. The lattices $\mathbf L(\mathbf Q)$ and $\mathbf L_C(\mathbf Q)$ are visualized in Figure~3 and 4, respectively:

\vspace*{-3mm}

\begin{center}
\setlength{\unitlength}{7mm}
\begin{picture}(11,14)
\put(5,1){\circle*{.3}}
\put(2,3){\circle*{.3}}
\put(4,3){\circle*{.3}}
\put(6,3){\circle*{.3}}
\put(8,3){\circle*{.3}}
\put(4,5){\circle*{.3}}
\put(6,5){\circle*{.3}}
\put(8,5){\circle*{.3}}
\put(10,5){\circle*{.3}}
\put(2,7){\circle*{.3}}
\put(4,7){\circle*{.3}}
\put(6,7){\circle*{.3}}
\put(1,9){\circle*{.3}}
\put(3,9){\circle*{.3}}
\put(5,9){\circle*{.3}}
\put(7,9){\circle*{.3}}
\put(3,11){\circle*{.3}}
\put(5,11){\circle*{.3}}
\put(7,11){\circle*{.3}}
\put(5,13){\circle*{.3}}
\put(5,1){\line(-3,2)3}
\put(5,1){\line(-1,2)3}
\put(5,1){\line(1,2)1}
\put(5,1){\line(3,2)3}
\put(2,3){\line(-1,6)1}
\put(2,3){\line(0,1)4}
\put(2,3){\line(1,6)1}
\put(4,3){\line(1,1)2}
\put(6,3){\line(-1,1)2}
\put(6,3){\line(0,1)4}
\put(6,3){\line(1,1)2}
\put(8,3){\line(0,1)2}
\put(8,3){\line(1,1)2}
\put(4,5){\line(-3,4)3}
\put(4,5){\line(0,1)2}
\put(6,5){\line(-1,1)2}
\put(6,5){\line(0,1)2}
\put(8,5){\line(-1,1)2}
\put(8,5){\line(-1,4)1}
\put(10,5){\line(-7,4)7}
\put(10,5){\line(-3,4)3}
\put(2,7){\line(1,4)1}
\put(4,7){\line(-1,4)1}
\put(4,7){\line(1,2)1}
\put(6,7){\line(-1,2)1}
\put(1,9){\line(1,1)4}
\put(3,9){\line(1,1)2}
\put(5,9){\line(1,1)2}
\put(7,9){\line(0,1)2}
\put(5,11){\line(0,1)2}
\put(7,11){\line(-1,1)2}
\put(4,5){\line(3,4)3}
\put(1,9){\line(2,1)4}
\put(4.75,.3){$P_1$}
\put(1.25,2.85){$P_5$}
\put(3.25,2.85){$P_2$}
\put(6.2,2.85){$P_4$}
\put(8.2,2.85){$P_3$}
\put(3.25,4.85){$P_9$}
\put(6.2,4.85){$P_6$}
\put(8.2,4.85){$P_7$}
\put(10.2,4.85){$P_8$}
\put(1,6.85){$P_{12}$}
\put(3.05,6.85){$P_{11}$}
\put(6.2,6.85){$P_{10}$}
\put(0,8.85){$P_{16}$}
\put(2,8.85){$P_{15}$}
\put(5.2,8.85){$P_{13}$}
\put(7.2,8.85){$P_{14}$}
\put(2,10.85){$P_{18}$}
\put(5.2,10.85){$P_{19}$}
\put(7.2,10.85){$P_{17}$}
\put(4.6,13.4){$P_{20}$}
\put(4.3,-.75){{\rm Fig.~3}}
\put(2.3,-1.75){Lattice $\mathbf L(\mathbf N_5\times[0,a])$}
\end{picture}
\end{center}

\vspace*{8mm}

\begin{center}
\setlength{\unitlength}{7mm}
\begin{picture}(6,8)
\put(3,1){\circle*{.3}}
\put(1,3){\circle*{.3}}
\put(3,3){\circle*{.3}}
\put(5,3){\circle*{.3}}
\put(1,5){\circle*{.3}}
\put(3,5){\circle*{.3}}
\put(5,5){\circle*{.3}}
\put(3,7){\circle*{.3}}
\put(3,1){\line(-1,1)2}
\put(3,1){\line(0,1)2}
\put(3,1){\line(1,1)2}
\put(3,7){\line(-1,-1)2}
\put(3,7){\line(0,-1)2}
\put(3,7){\line(1,-1)2}
\put(3,3){\line(-1,1)2}
\put(3,3){\line(1,1)2}
\put(3,5){\line(-1,-1)2}
\put(3,5){\line(1,-1)2}
\put(1,3){\line(0,1)2}
\put(5,3){\line(0,1)2}
\put(2.75,.3){$P_1$}
\put(.25,2.85){$P_5$}
\put(3.2,2.85){$P_2$}
\put(5.2,2.85){$P_8$}
\put(0,4.85){$P_{12}$}
\put(3.2,4.85){$P_{15}$}
\put(5.2,4.85){$P_{17}$}
\put(2.6,7.4){$P_{20}$}
\put(2.3,-.75){{\rm Fig.~4}}
\put(.3,-1.75){Lattice $\mathbf L_C(\mathbf N_5\times[0,a])$}
\end{picture}
\end{center}

\vspace*{10mm}

The mappings $^\perp$ and $^{\perp\perp}$ on $L(\mathbf Q)$ are described by the following tables:
\[
\begin{array}{l|l|l|l|l|l|l|l|l|l|l}
P              & P_1    & P_2    & P_3    & P_4    & P_5    & P_6    & P_7    & P_8    & P_9    & P_{10} \\
\hline
P^\perp        & P_{20} & P_{15} & P_{12} & P_5    & P_{17} & P_5    & P_5    & P_{12} & P_5    & P_5 \\
\hline
P^{\perp\perp} & P_1    & P_2    & P_8    & P_{17} & P_5    & P_{17} & P_{17} & P_8    & P_{17} & P_{17}
\end{array}
\]
\[
\begin{array}{l|l|l|l|l|l|l|l|l|l|l}
P              & P_{11} & P_{12} & P_{13} & P_{14} & P_{15} & P_{16} & P_{17} & P_{18} & P_{19} & P_{20} \\
\hline
P^\perp        & P_5    & P_8    & P_5    & P_5    & P_2    & P_1    & P_5    & P_1    & P_1    & P_1 \\
\hline
P^{\perp\perp} & P_{17} & P_{12} & P_{17} & P_{17} & P_{15} & P_{20} & P_{17} & P_{20} & P_{20} & P_{20}
\end{array}
\]
Similarly, the restrictions of these mappings to $L_C(\mathbf Q)$ are described in the table below:
\[
\begin{array}{l|l|l|l|l|l|l|l|l}
P              & P_1    & P_2    & P_5    & P_8    & P_{12} & P_{15} & P_{17} & P_{20} \\
\hline
P^\perp        & P_{20} & P_{15} & P_{17} & P_{12} & P_8    & P_2    & P_5    & P_1 \\
\hline
P^{\perp\perp} & P_1    & P_2    & P_5    & P_8    & P_{12} & P_{15} & P_{17} & P_{20} \\
\end{array}
\]
\end{example}

The next result solves the problem when the mapping $^{\perp\perp}$ is a homomorphism from $\mathbf L(\mathbf Q)$ onto $\mathbf L_C(\mathbf Q)$.

\begin{theorem}
Let $\mathbf Q=(Q,+,\cdot,\vec0)=\prod\limits_{i\in I}\mathbf L_i$ be a canonical quasimodule with $0$-distributive $\mathbf L_i$. Then the following hold:
\begin{enumerate}[{\rm(i)}]
\item If $(P\cap R)^{\perp\perp}=P^{\perp\perp}\cap R^{\perp\perp}$ for all $P,R\in L(\mathbf Q)$ then $^{\perp\perp}$ is a homomorphism from $\big(L(\mathbf Q),\vee,\cap,{}^\perp,\{0\},Q\big)$ onto $\big(L_C(\mathbf Q),\stackrel c\vee,\cap,{}^\perp,\{0\},Q\big)$.
\item If $\left(\bigcap\limits_{j\in J}P_j\right)^{\perp\perp}=\bigcap\limits_{j\in J}P_j^{\perp\perp}$ for every family $P_j,j\in J,$ of subquasimodules of $\mathbf Q$ then $^{\perp\perp}$ is a complete homomorphism from $\big(L(\mathbf Q),\vee,\cap,{}^\perp,\{0\},Q\big)$ onto $\big(L_C(\mathbf Q),\stackrel c\vee,\cap,{}^\perp,\{0\},Q\big)$, that is, $\left(\bigvee\limits_{j\in J}P_j\right)^{\perp\perp}=\bigvee\limits_{j\in J}^cP_j^{\perp\perp}$ and $\left(\bigcap\limits_{j\in J}P_j\right)^{\perp\perp}=\bigcap\limits_{j\in J}P_j^{\perp\perp}$ for every family $P_j,j\in J,$ of subquasimodules of $\mathbf Q$.
\end{enumerate}
\end{theorem}

\begin{proof}
Let $P,R,P_j\in L(\mathbf Q)$ for every $j\in J$.
\begin{enumerate}[(i)]
\item We have
\begin{align*}
(P\vee R)^{\perp\perp} & =(P\cup R)^{\perp\perp}=(P^\perp\cap R^\perp)^\perp=P^{\perp\perp}\stackrel c\vee R^{\perp\perp}, \\
(P\cap R)^{\perp\perp} & =P^{\perp\perp}\cap R^{\perp\perp}, \\
(P^\perp)^{\perp\perp} & =(P^{\perp\perp})^\perp, \\
    \{0\}^{\perp\perp} & =\{0\}, \\
        Q^{\perp\perp} & =Q.
\end{align*}
\item We have
\begin{align*}
\left(\bigvee_{j\in J}P_j\right)^{\perp\perp} & =\left(\bigcup_{j\in J}P_j\right)^{\perp\perp}=\left(\bigcap_{j\in J}P_j^\perp\right)^\perp=\bigvee_{j\in J}^cP_j^{\perp\perp}, \\
\left(\bigcap_{j\in J}P_j\right)^{\perp\perp} & =\bigcap_{j\in J}P_j^{\perp\perp}, \\
                       (P^\perp)^{\perp\perp} & =(P^{\perp\perp})^\perp, \\
                           \{0\}^{\perp\perp} & =\{0\}, \\
                               Q^{\perp\perp} & =Q.
\end{align*}
\end{enumerate}
\end{proof}

We are interested in the case when for a subquasimodule $(P,+,\cdot,\vec0)$ of a quasimodule $\mathbf Q$, $P^\perp$ is a complement of $P$ in the lattice $\mathbf L(\mathbf Q)$. For this, we introduce the following concept.

\begin{definition}
Let $\mathbf Q=(Q,+,\cdot,\vec0)$ be a canonical quasimodule. A subset $A$ of $Q$ is called {\em splitting} if $A+A^\perp=Q$. Here $A+A^\perp$ means $\{\vec x+\vec y\mid\vec x\in A,\vec y\in A^\perp\}$. Let $L_S(\mathbf Q)$ denote the set of all splitting subquasimodules of $\mathbf Q$ and put $\mathbf L_S(\mathbf Q):=\big(L_S(\mathbf Q),\subseteq\big)$.
\end{definition}

Note that according to (iv) of Lemma~\ref{lem4} we have $A\cap A^\perp=\{0\}$ for every non-empty subset $A$ of $Q$ and hence a subquasimodule $(P,+,\cdot,\vec0)$ of $\mathbf Q$ is splitting if and only if both $P+P^\perp=Q$ and $P\cap P^\perp=\{0\}$.

\begin{example}
Consider the canonical quasimodule $\mathbf Q:=\mathbf N_5\times\{0,a\}$ from Example~\ref{ex1}. Then $L_S(\mathbf Q)=L_C(\mathbf Q)$.
\end{example}

However, the situation from the previous example where splitting subquasimodules coincide with closed ones does not hold in general. We can prove only that every splitting subquasimodule is closed.

\begin{theorem}
Let $\mathbf Q=\prod\limits_{i\in I}\mathbf L_i$ be a canonical quasimodule. Then $L_S(\mathbf Q)\subseteq L_C(\mathbf Q)$.
\end{theorem}

\begin{proof}
Assume $L_S(\mathbf Q)\not\subseteq L_C(\mathbf Q)$. Then there exists some $P\in L_S(\mathbf Q)\setminus L_C(\mathbf Q)$. Since $P$ is not closed there exists some $\vec x\in P^{\perp\perp}\setminus P$. Because $P$ is splitting there exist $\vec y\in P$ and $\vec z\in P^\perp$ with $\vec y+\vec z=\vec x$. Let $x_i,i\in I,$ denote the components of $\vec x$. Analogous notations we use for the components of $\vec y$ and $\vec z$. Since $\vec y+\vec z=\vec x$ we have $y_i\vee z_i=x_i$ for all $i\in I$. Because $\vec x\in P^{\perp\perp}$ and $\vec z\in P^\perp$ we have $\vec x\perp\vec z$, that is, $x_i\wedge z_i=0$ for all $i\in I$. Together we obtain $z_i=(y_i\vee z_i)\wedge z_i=x_i\wedge z_i=0$ for all $i\in I$, that is, $\vec z=\vec0$ which implies $\vec x=\vec y+\vec z=\vec y\in P$, a contradiction. This shows $L_S(\mathbf Q)\subseteq L_C(\mathbf Q)$.
\end{proof}

That the converse inclusion does not hold in general is demonstrated by the following example.

\begin{example}
Let $\mathbf L$ denote the non-modular $0$-distributive lattice depicted in Figure~5:

\vspace*{-3mm}
	
\begin{center}
\setlength{\unitlength}{7mm}
\begin{picture}(4,10)
\put(2,1){\circle*{.3}}
\put(1,4){\circle*{.3}}
\put(3,3){\circle*{.3}}
\put(3,5){\circle*{.3}}
\put(2,7){\circle*{.3}}
\put(2,9){\circle*{.3}}
\put(2,1){\line(-1,3)1}
\put(2,1){\line(1,2)1}
\put(3,3){\line(0,1)2}
\put(2,7){\line(-1,-3)1}
\put(2,7){\line(1,-2)1}
\put(2,7){\line(0,1)2}
\put(1.85,.3){$0$}
\put(3.4,2.85){$a$}
\put(.35,3.85){$b$}
\put(3.4,4.85){$c$}
\put(2.4,6.85){$d$}
\put(1.85,9.4){$1$}
\put(1.2,-.75){{\rm Fig.~5}}
\put(-2.4,-1.75){Non-modular $0$-distributive lattice $\mathbf L$}
\end{picture}
\end{center}

\vspace*{10mm}

Consider the canonical quasimodule $\mathbf Q:=\mathbf L\times\mathbf L$ over $\mathbf L$. Then $P:=[0,b]\times[0,c]$ is a closed subquasimodule of $\mathbf Q$ since $[0,b]$ and $[0,c]$ are closed subquasimodules of the canonical quasimodule $\mathbf L$ over $\mathbf L$, but $P$ is not splitting since $(1,1)\notin P+P^\perp=([0,b]\times[0,c])+([0,c]\times[0,b])$.
\end{example}

In the following proposition we show that for every splitting subquasimodule $(P,+,\cdot,\vec0)$ of a canonical quasimodule $\mathbf Q$ with $0$-distributive components its orthogonal companion $(P^\perp,+,\cdot,\vec0)$ is again splitting.

\begin{proposition}
Let $\mathbf Q=(Q,+,\cdot,\vec0)=\prod\limits_{i\in I}\mathbf L_i$ be a canonical quasimodule with $0$-distributive $\mathbf L_i$. Then the following hold:
\begin{enumerate}[{\rm(i)}]
\item For every splitting subquasimodule $(P,+,\cdot,\vec0)$ of $\mathbf Q$ the subquasimodule $(P^\perp,+,\cdot,\vec0)$ of $\mathbf Q$ is again splitting.
\item The mapping $^\perp$ is an antitone mapping from $\mathbf L_S(\mathbf Q)$ to $\mathbf L_C(\mathbf Q)$.
\end{enumerate}
\end{proposition}

\begin{proof}
Let $P\in L_S(\mathbf Q)$.
\begin{enumerate}[(i)]
\item We have $P+P^\perp=Q$. According to (i) of Remark~\ref{rem1} we conclude $P^\perp+P^{\perp\perp}=Q$ showing $P^\perp\in P_S(\mathbf Q)$.
\item According to (i) of Theorem~\ref{th2} we have $P^\perp\in L_C(\mathbf Q)$. The rest of the proof follows from (ii) of Remark~\ref{rem1}.
\end{enumerate}
\end{proof}

The final proposition shows that for a canonical quasimodule $\mathbf Q=\prod\limits_{i\in I}\mathbf L_i$ with $\mathbf L_i=(L_i,\vee,\wedge,0)$ for all $i\in I$ a direct product of certain subsets of $L_i$ is a splitting subquasimodule of $\mathbf Q$ if and only if its factors are splitting subquasimodules of $\mathbf L_i$.

\begin{proposition}
Let $\mathbf Q=(Q,+,\cdot,\vec0)=\prod\limits_{i\in I}\mathbf L_i$ with $\mathbf L_i=(L_i,\vee,\wedge,0)$ be a canonical quasimodule and $M_i\subseteq L_i$ for all $i\in I$. Then $\prod\limits_{i\in I}M_i\in L_S(\mathbf Q)$ if and only if $M_i\in L_S(\mathbf L_i)$ for all $i\in I$.
\end{proposition}

\begin{proof}
Because of Lemma~\ref{lem6} we have $\prod\limits_{i\in I}M_i\in L(\mathbf Q)$ if and only if $M_i\in L(\mathbf L_i)$ for all $i\in I$. Now assume $\prod\limits_{i\in I}M_i\in L(\mathbf Q)$. Then according to Lemma~\ref{lem1} the following are equivalent:
\begin{align*}
                                     \prod_{i\in I}M_i & \in L_S(\mathbf Q), \\	
\prod_{i\in I}M_i+\left(\prod_{i\in I}M_i\right)^\perp & =Q, \\
             \prod_{i\in I}M_i+\prod_{i\in I}M_i^\perp & =Q, \\
                     \prod_{i\in I}(M_i\vee M_i^\perp) & =Q, \\
                                     M_i\vee M_i^\perp & =L_i\text{ for all }i\in I, \\
                                                       M_i & \in L_s(\mathbf L_i)\text{ for all }i\in I.
\end{align*}
\end{proof}








Authors' addresses:

Ivan Chajda \\
Palack\'y University Olomouc \\
Faculty of Science \\
Department of Algebra and Geometry \\
17.\ listopadu 12 \\
771 46 Olomouc \\
Czech Republic \\
ivan.chajda@upol.cz

Helmut L\"anger \\
TU Wien \\
Faculty of Mathematics and Geoinformation \\
Institute of Discrete Mathematics and Geometry \\
Wiedner Hauptstra\ss e 8-10 \\
1040 Vienna \\
Austria, and \\
Palack\'y University Olomouc \\
Faculty of Science \\
Department of Algebra and Geometry \\
17.\ listopadu 12 \\
771 46 Olomouc \\
Czech Republic \\
helmut.laenger@tuwien.ac.at
\end{document}